\documentclass{amsart}
\usepackage{latexsym,amsmath,amssymb,amscd}

\newtheorem{theorem}{Theorem}

\theoremstyle{definition}

\theoremstyle{remark}

\numberwithin{equation}{section}

\def\fil{\Phi}

\def\rfi{Ran(\Phi)}

\begin{document}

\title{On the Choi-Effros multiplication}
\author{Bebe Prunaru}

\address{Institute of Mathematics ``Simion Stoilow''
  of the Romanian
Academy,
P.O. Box 1-764,
 RO-014700 Bucharest, 
Romania}
\email{Bebe.Prunaru@imar.ro}


\begin{abstract}

A short proof is given for the 
well-known Choi-Effros theorem on the structure 
of ranges of completely positive 
projections.

\end{abstract}

\maketitle

In this note an alternate proof is given 
for the following well-known 
and basic 
theorem of M. D. Choi and E. G. Effros 
(see Theorem 3.1 in \cite{CE}):

\begin{theorem}

Let $A$ be a 
$C^*$-algebra and let 
$\fil:A\to A$ 
be a completely positive, contractive and 
idempotent linear map.
Then there exist a 
  $C^*$-algebra  $B$  and a complete 
order isomorphism
$\rho:B\to\rfi$
such that 
$\rho(ab)=\fil(\rho(a)\rho(b))$
for all $a,b\in B.$

\end{theorem}

\begin{proof}

We may assume that $A$ is 
generated, 
as a $C^*$-algebra, by $\rfi$.
Let $J$ denote the closed right  
ideal of $A$ 
 generated by all operators 
of the form
$xy-\fil(xy)$
with $x,y\in\rfi$.
We will show that $Ker(\fil)=J$.
Let $z\in Ker(\fil)$    with $z\ge 0$ and let 
$y\in A$.
Then, by the Kadison-Schwarz
 inequality,
$$\fil(zy)\fil(zy)^*\le \|z^{1/2}y\|^2
\fil(z)=0.$$
In particular this holds true 
 when $z=\fil(x^*x)-x^*x$ for some
  $x\in\rfi$.
 This shows that $J\subset Ker(\fil)$.
 We will now show, by induction over $k$, 
 that if 
 $u=x_1\cdots x_k$
 with $x_j\in\rfi$ for 
  $j=1,\dots,k$
 then $u-\fil(u)\in J$.
 When $k=1$ or $k=2$ this is obvious.
 Suppose that $k\ge3$ and
  assume
  it holds for $k-1$. 
  Write 
  $u=u_1+u_2$ where
  $$u_1=(x_1x_2-\fil(x_1x_2))x_3\cdots
 x_k$$
  and
  $$u_2=\fil(x_1x_2)x_3\cdots x_k.$$
  Then $u_1\in J$ hence 
  $\fil(u_1)=0$,
  and $u_2-\fil(u_2)\in J$ 
  by the induction hypothesis.
  It then follows that $Ker(\fil)=J$
  and, in particular, that $Ket(\fil)$
  is a bilateral  ideal in $A$.
  Let $B=A/Ker(\fil)$  and let 
  $\rho:B\to\rfi$ 
  be the quotient map
   induced by $\fil$. 
  It is now routine to see that
  the couple $(B,\rho)$ 
  has all the properties we were looking for.

\end{proof}


\begin{thebibliography}{10000000}

\bibitem{CE}
M.D.~Choi, E.G.~Effros,
\textit{Injectivity and operator spaces}, 
J. Functional Analysis , \textbf{24}  (1977), no.~2, 156--209.



\end{thebibliography}
\end{document}